\documentclass[a4paper,reqno]{amsart}

\usepackage{amsmath}
\usepackage{paralist}
\usepackage{graphics}
\usepackage{epsfig}
\usepackage{amsfonts}
\usepackage{amssymb}
\usepackage{hyperref}
\usepackage{epstopdf}
\usepackage{color}

\newtheorem{theorem}{Theorem}[section]
\newtheorem{lemma}[theorem]{Lemma}

\def\ifl{\iffalse }

\def\bc{\begin{center}}       \def\ec{\end{center}}
\def\ba{\begin{array}}        \def\ea{\end{array}}
\def\be{\begin{equation}}     \def\ee{\end{equation}}
\def\bea{\begin{eqnarray}}    \def\eea{\end{eqnarray}}
\def\beaa{\begin{eqnarray*}}  \def\eeaa{\end{eqnarray*}}

\numberwithin{equation}{section}

\newtheorem{remark}[theorem]{Remark}
\numberwithin{equation}{section}

\begin{document}

\title[minimal Chemotaxis system, sub-logistic source, and boundedness]
{Sub-logistic source can prevent blow-up in the 2D  minimal  Keller-Segel chemotaxis system}

\author{Tian Xiang}
\address{Institute for Mathematical Sciences, Renmin University of China, Beijing,  100872, China}
\email{txiang@ruc.edu.cn}
\thanks{This project was partially supported by the NSF of China (No. 11601516).}




\begin{abstract}

It is well-known that the Neumann initial-boundary value problem for the minimal-chemotaxis-logistic system
$$\left\{ \begin{array}{lll}
&u_t = \Delta u-\chi\nabla \cdot ( u\nabla v)+a u-b u^2,  &\quad x\in \Omega, t>0, \\[0.2cm]
& \tau v_t =  \Delta v  -  v+  u,  &\quad x\in \Omega, t>0 \end{array}\right.
$$
 in a bounded smooth   domain  $\Omega\subset \mathbb{R}^2$ doesn't have any blow-ups for any $a\in\mathbb{R}, \tau\geq 0, \chi>0$ and $b>0$. Here, we obtain the same conclusion by replacing the logistic source $a u-b u^2$ with a kinetic term $f(u)$ fulfilling $f(0)\geq 0$,
 $$
 \liminf_{s\to \infty}\Bigr\{-f(s)\cdot \frac{\ln s}{s^2}\Bigr\}=\mu\in (0, \infty]
$$
as well as
$$
(\chi-\mu)^+M<\frac{1}{2C_{GN}^4},
$$
where $c^+=\max\{c,0\}$, $C_{GN}$ is the Gagliardo-Nirenberg  constant and
$$
  M=\|u_0\|_{L^1(\Omega)}+|\Omega|\inf_{\eta>0}\frac{\sup\{ f(s)+\eta s:\  s>0\}}{\eta}.
$$
In this setup, it is shown that this problem doesn't have any blow-up by ensuring all solutions  are global-in-time and uniformly bounded. Clearly, $f$ covers super-, logistic, sub-logistic sources like $f(s)=as-bs^\theta$ with $b>0$ and $\theta\geq2$, $f(s)=as-\frac{bs^2}{\ln^\gamma(s+1)}$ with $b>0$ and $\gamma\in (0,1)$, and $f(s)=as-\frac{bs^2}{\ln(\ln(s+e))}$ with $b>0$ etc. This indicates that  logistic damping is not  the weakest damping  to guarantee boundedness for the 2D Keller-Segel minimal chemotaxis  model.
\end{abstract}

\maketitle

\section{Introduction}

  Chemotaxis is the directed  movement of mobile species in response to chemical signal in their environment. To model this important process,  in 1970s,  Keller-Segel \cite{Ke, Ke1} proposed a classical  coupled parabolic partial differential system to describe the mobile cells (with density $u$) move towards the concentration gradient of a chemical substance $v$ produced by the cells
themselves.  Nowadays, this system is widely  known as the minimal Keller-Segel chemotaxis model:
  \be \label{KSM}\begin{cases}
u_t = \Delta u-\chi\nabla \cdot ( u\nabla v),  &x\in \Omega, t>0, \\[0.2cm]
\tau v_t =  \Delta v- v+ u,  & x\in \Omega, t>0, \\[0.2cm]
\frac{\partial u}{\partial \nu}=\frac{\partial v}{\partial \nu}=0, & x\in \partial \Omega, t>0,\\[0.2cm]
u(x,0)=u_0(x)\geq 0,  \tau v(x,0)= \tau v_0(x)\geq 0, & x\in \Omega,    \end{cases}  \ee
where the habitat $\Omega\subset \mathbb{R}^n(n\geq 1)$ is a bounded domain with the smooth boundary $\partial \Omega$, $\tau\geq 0, \chi>0$ and  $\frac{\partial}{\partial\nu}$ stands for the outward  normal derivative on $\partial\Omega$.

Since this pioneering work, the minimal model \eqref{KSM} and it various variants have been  received significant attention  to understand chemotaxis mechanism in various contexts and, the chemotactic induced cross-diffusion  has been shown to lead to finite/infinite-time blow-up under certain circumstances. The grand phenomenological picture is surveyed as follows:  No finite/infinite  time blow-up occurs in $1$-D  \cite{HP04, OY01, Xiang14},   critical  mass blow-ups occur in $2$-D: when the initial mass  lies below the threshold solutions exist globally and converge to a single equilibrium in large time, while above the threshold solutions may blow up in finite time, cf. \cite{HW01,FLP07, NSY97, Na01, SS01}, and even small initial mass will result in  blow-ups in $\geq 3$-D \cite{Win10-JDE, Win13}.  Accordingly, a large variety of work has been dedicated to determining  the situations where either  blow-ups occur or  global-in-time bounded solutions exist. See the review articles  \cite{Ho1,Hi, Win13, BBTW15} for more  progresses on  \eqref{KSM}  and its variants.

To see that the chemotactic induced cross-diffusion  $-\chi\nabla\cdot(u\nabla v)$ in model  \eqref{KSM} has an aggregation effect, in the parabolic-elliptic case, i.e., $\tau=0$, let  us expand out the $u$-equation and then substitute the expression of $\Delta v=v-u$ obtained from the $v$-equation, we see that the $u$-equation becomes
\be\label{u-sub-pe}
u_t = \Delta  u-\chi \nabla u \nabla v+\chi u^2-\chi uv.
\ee
This  tells us that the chemotaxis term $-\chi\nabla\cdot(u\nabla v)$ does play a  chemotactic aggregation role and it behaves roughly like the quadratic growth $\chi u^2$,   and, rigorous if $v$ is $L^\infty$-bounded. Therefore, adding a logistic source of the form $a u-b u^2$ with $a \in\mathbb{R}$ and $b>0$ to the $u$-equation in \eqref{KSM} is very much expected to eliminate such finite/inifinte time blow-up phenomenon. The presence of logistic source indeed has been demonstrated to have an effect  of preventing  blow-ups for \eqref{KSM}. More specifically,  for the minimal-chemotaxis-logistic model (with the same boundary and initial data as \eqref{KSM}  suspended for less writing and saving space)
\begin{equation}\label{minimal-model}\left\{ \begin{array}{lll}
&u_t = \Delta u-\chi\nabla \cdot ( u\nabla v)+a u-b u^2,  &\quad x\in \Omega, t>0, \\[0.2cm]
&\tau v_t = \Delta  v  - v+u,  &\quad x\in \Omega, t>0,  \end{array}\right.  \end{equation}
in the case of $\tau=0$, if $b>\chi$, then a replacement of  the $u$-equation in \eqref{minimal-model} with Eq. \eqref{u-sub-pe}  roughly  shows that every solution to \eqref{minimal-model} with $\tau=0$ is global-in-time and is uniformly bounded. In fact, subtle and rigorous studies, cf. \cite{OY01, OTYM02, HP04, Xiangpre}, have shown, for $n=1,2$, no matter $\tau=0$ or $\tau>0$,   {\it even arbitrarily  small $b>0$ } will be enough to rule out any blow-up by ensuring all solutions to  \eqref{minimal-model} are global-in-time and uniformly bounded for all reasonably initial data. This  is even  true in respective of  global existence of classical solutions  for \eqref{minimal-model} with $\tau=0$ and  the chemo-sensitivity  $-\chi \nabla\cdot(u\nabla v)$ replaced by a singular chemo-sensitivity  $-\chi \nabla\cdot(\frac{u}{v}\nabla v)$ and uniform boundedness is further ensured in \cite{FW14} for sufficiently large $a>0$. The same result has been recently extended to the fully parabolic case \cite{ZZ17-ZAMP}. These results convey to us that there is no difference between parabolic-elliptic case ($\tau=0$) and parabolic-parabolic case ($\tau>0$) in $2$-D in respective of global existence and boundedness for \eqref{minimal-model}. In $2$-D setting, comparing the results for \eqref{minimal-model} with that of \eqref{KSM}, we find that blow-up is fully precluded  as long as a logistic source presents, and, in this case, there is no critical mass blow-up phenomenon.   Based on these observations, we wonder
\begin{itemize}
\item[(Q)] adding a logistic source may be more than enough to prevent blow-up for \eqref{KSM} in $2$-D, and thus we   wonder  whether or not adding a sub-logistic source is already sufficient to prevent blow-up for \eqref{KSM} in $2$-D?
\end{itemize}
In this paper,  we obtain a positive answer  to  (Q) by showing  that a sub-logistic source is already enough to prevent blow-up for \eqref{KSM} in $2$-D. To state our precise results, we replace the logistic source in \eqref{minimal-model} with a  growth source $f$  or equivalently we add $f$ to the $u$-equation in \eqref{KSM} and then consider the resulting IBVP:
\be\label{sub-log-chemo-pe}\left\{ \begin{array}{lll}
&u_t = \Delta u-\chi\nabla \cdot (u \nabla v)+f(u),  &\quad x\in \Omega, t>0, \\[0.2cm]
&\tau v_t = \Delta  v  -  v+ u,  &\quad x\in \Omega, t>0, \\[0.2cm]
&\frac{\partial u}{\partial \nu}=\frac{\partial v}{\partial \nu}=0, &\quad x\in \partial \Omega, t>0,\\[0.2cm]
&u(x,0)=u_0(x)\geq,\not\equiv 0, \tau v(x,0)=\tau v_0(x)\geq 0, &\quad x\in \Omega.  \end{array}\right.  \ee
 With this setup,  our precise  findings on blow-up prevention by sub-logistic sources for \eqref{KSM} in $2$-D read as follows:
\begin{theorem} \label{main-bdd-thm} [Blow-up prevention by sub-logistic sources  for \eqref{sub-log-chemo-pe} in 2-D] Let  $\Omega\subset\mathbb{R}^2$ be a bounded smooth domain, $\tau\geq 0$, $\chi>0$,  the initial data  $(u_0,v_0)$ satisfy  $u_0\in C(\overline{\Omega})$  and $v_0\in W^{1,r}(\Omega)$ for some $r>2$ in the case of $\tau>0$, and, finally,  let the kinetic  function $f$ belonging to $W^{1,\infty}_{\text{loc}}(\Omega)$ satisfy  $f(0)\geq 0$ as well as
\be\label{log-sub-con}
\liminf_{s\to \infty}\Bigr\{-f(s)\cdot \frac{\ln s}{s^2}\Bigr\}=\mu\in (0, \infty].
\ee
Assume either one of the following cases holds:
\begin{itemize}
\item[(B1)] Sub-logistic, logistic or super-logistic source, i.e.,  $\mu=\infty$;
\item[(B2)] proper sub-logistic source dominates or cancels   chemotactic aggregation, i.e., $0<\mu<\infty$ and $\mu\geq \chi$;
\item[(B3)]chemotactic aggregation dominates proper sub-logistic source and small initial mass, i.e., $\chi>\mu$,  $0<\mu<\infty$  and
 $$
(\chi-\mu)M<\frac{1}{2C_{GN}^4},
$$
where  $C_{GN}$ is the Gagliardo-Nirenberg  constant, $M$ is finite and is given by
\be\label{M-def}
  M=\|u_0\|_{L^1(\Omega)}+|\Omega|\inf_{\eta>0}\frac{\sup\{ f(s)+\eta s:\  s>0\}}{\eta}.
\ee
\end{itemize}
Then the  Keller-Segel chemotaxis-growth   model \eqref{sub-log-chemo-pe} has a unique global-in-time classical solution which is uniformly-in-time bounded in the following ways: there  exists a constant  $C=C(u_0,\tau v_0, |\Omega|, \chi, f)>0$   such that
\be\label{bdd-linfty}
\|u(\cdot,t)\|_{L^\infty(\Omega)}+\|v(\cdot,t)\|_{W^{1,r}(\Omega)}\leq C, \quad \quad  \forall t\in (0,\infty)
\ee
and, for any $\sigma>0$, there exists $C_\sigma=C(\sigma, u_0,\tau v_0, |\Omega|, \chi, f)>0$ such that
\be\label{bdd-linfty1}
\|u(\cdot,t)\|_{L^\infty(\Omega)}+\|v(\cdot,t)\|_{W^{1,\infty}(\Omega)}\leq C, \quad \quad  \forall t\in (\sigma,\infty).
\ee
\end{theorem}
\begin{remark}\label{main-thm-rem} [Notes on blow-up prevention by sub-logistic sources in 2-D] Here and below, $\mu$ is understood as an extended real positive number. Hence, via the positive part function $a^+=\max\{a,0\}$, the cases (B1), (B2) and (B3) can be unified as
 \be\label{sub-log-general}
(\chi-\mu)^+M<\frac{1}{2C_{GN}^4}.
\ee
 The case that $\mu=\infty$ covers sub-logistic, logistic or super-logistic sources; typical examples are $f(s)=as-\frac{bs^2}{\ln^\gamma (s+1)}$ with $a\in \mathbb{R}, b>0, \gamma\in(0,1)$ for $s>0$, or $f(s)=as-\frac{bs^2}{\ln(\ln(s+e))}$ with $b>0$ etc (sub-logistic sources); $f(s)=a s-bs^\theta$ with $a\in \mathbb{R}$,$ b>0$ and $\theta\geq 2$ (logistic or super-logistic sources). In the last case, the boundedness of solutions to \eqref{sub-log-chemo-pe} in 2-D is well-known as mentioned before.

  We note also  that not all sub-logistic sources are included in Theorem \ref{main-bdd-thm}; for instance, growth sources like $f(s)=as-bs^\alpha$ with $a\in \mathbb{R}, b>0$ and $1<\alpha<2$ fail to fulfill \eqref{log-sub-con} since then  $\mu=0$. We leave this challenging problem (cf. remarks below)  as future research to explore  whether  or not the radical  opposite side of boundedness, namely, blow-up,  will occur for the 2-D chemotaxis  model  \eqref{sub-log-chemo-pe}. However, the main message that Theorem \ref{main-bdd-thm} conveys to us is that logistic damping is not the weakest damping to guarantee boundedness for the minimal chemotaxis-logistic   model  \eqref{sub-log-chemo-pe} in 2-D.
\end{remark}
 In higher dimensions,  the logistic effect on boundedness  becomes increasingly complex.   For  $n\geq 3$,  so far, it is only known that properly strong logistic damping can prevent blow-up for \eqref{KSM}. More specifically,  in the  case of $\tau=0$,  all classical solutions to  \ref{minimal-model} are uniformly bounded in time if
\be\label{mu-strong0}
  b\geq \frac{(n-2)}{n}\chi.
\ee
See, for instance, \cite{TW07, HT16, KS16, WXpre}.

  In the  parabolic-parabolic case ($\tau>0$), the issue becomes more delicate; for $n\geq 4$,  sufficiently strong logistic damping can prevent blow-up for \eqref{KSM} (cf. \cite{Win10, YCJZ15}), and, in the physically relevant case of  $n=3$, we have a compact formula:
  \be\label{mu-strong}
b> \left\{ \begin{array}{lll}
&\frac{3}{4}\chi, & \text{if }  \tau=1  \text{  and  }  \Omega \text{  is convex  \cite{Win10}},   \\[0.25cm]
&\frac{3(1+2\tau)}{(\sqrt{10}-2)\tau}\chi,  & \text{if  either }  \tau\neq 1  \text{  or   }  \Omega \text{  is  not convex  \cite{Xiangpre2}}.  \end{array}\right.
\ee
\begin{remark} From  Theorem \ref{main-bdd-thm} and the text before it, we see that the 2-D boundedness doesn't distinguish  between parabolic-elliptic case and parabolic-parabolic case. However, by \eqref{mu-strong0} and \eqref{mu-strong}, we see that higher-dimensional boundedness  distinguishes drastically between parabolic-elliptic case and parabolic-parabolic case (the convexity of $\Omega$  and the equality of diffusion rates of $u$ and $v$  also matter a lot).
\end{remark}
  For other aspects, we mention, in $3$-D bounded, smooth and convex domains, that logistic damping guarantees the existence of global weak solutions to \eqref{minimal-model}  \cite{La15-JDE}, and that sufficiently strong logistic damping can enhance  'expected' results such as global
existence and stabilization toward constant equilibria, as well as more
'unexpected' behavior witnessing a certain strength of chemotactic destabilization for \eqref{minimal-model}, cf.  \cite{Cao15, HP11, HZ16, La15-DCDS, TW15-JDE,WXpre,  Win14, Win14-JNS, Win17-DCDSB, Xiangpre2}.

 Finally,  we mention something that is related to Remark \ref{main-thm-rem}. One might think that blow-up may not be prevented by weak damping sources. However, since nice properties such as  energy-like structure possessed by \eqref{KSM} are destroyed by the presence of growth source,  blow-up  has not been rigorously detected to occur in any chemotaxis-growth system  when $n=2,3$. Up to now, only when $n\geq 5$, radially symmetrical blow-up is known to be possible in a parabolic-elliptic simplified variant of \eqref{minimal-model} under a proper sub-quadratic damping source: for the parabolic-elliptic chemotaxis-growth system
$$
u_t=\Delta u-\nabla\cdot(u\nabla v)+au-bu^\alpha, \ \ 0=\Delta v+u-m(t), \ \ m(t)=\frac{1}{\Omega}\int_\Omega u,
$$
radially symmetrical blow-up may occur for space-dimension $n\geq 5$ and exponents $1<\alpha< \frac{3}{2}+\frac{1}{2n-2}$ \cite{Win11}.  This doesn't contradict Theorem \ref{main-bdd-thm} since $\mu=0$ by \eqref{log-sub-con}.  Thus, the opposite side of Theorem \ref{main-bdd-thm}, namely, the occurrence of blow-up  is a challenging problem.

The rest of this paper is arranged as follows: in Section 2, we collect the standard  Gagliardo-Nirenberg interpolation inequality and  the local well-posedness of \eqref{sub-log-chemo-pe}. In Section 3, we show the  proof of Theorem \ref{main-bdd-thm}; the key point consists in deriving a uniform-in-time estimate for $u\ln u$ under \eqref{sub-log-general}, afterward, we illustrate two commonly used  approaches, cf.  \cite{OTYM02, TW12, Xiangpre},  to establish the $L^2$-boundedness $u$, and then a simple application of the widely known  $L^{\frac{n}{2}+}$-boundedness criterion with $n=2$ (cf. \cite{BBTW15, Xiangpre}) obtained via Moser type iteration technique, we achieve the  global existence and  boundedness as in \eqref{bdd-linfty} and \eqref{bdd-linfty1} of solutions to \eqref{sub-log-chemo-pe}.
\section{Preliminaries}

For convenience, we state  the well-known Gagliardo-Nirenberg  inequality:
\begin{lemma}\label{GN-inter}(Gagliardo-Nirenberg interpolation inequality \cite{Fried,  Nirenberg66}) Let $p\geq 1$ and  $q\in (0,p)$. Then there exist a positive constant   $C_{GN}$ depending on $p$ and $q$ such that
 $$
 \|w\|_{L^p} \leq C_{GN}\Bigr(\|\bigtriangledown w\|_{L^2}^{\delta}\|w\|_{L^q}^{(1-\delta) }+\|w\|_{L^s}\Bigr), \quad \forall w\in H^1\cap L^q,
 $$
where  $s>0$ is arbitrary and  $\delta$ is given by
 $$
 \frac{1}{p}=\delta(\frac{1}{2}-\frac{1}{n})+\frac{1-\delta}{q}\Longleftrightarrow \delta=\frac{\frac{1}{q}-\frac{1}{p}}{\frac{1}{q}-\frac{1}{2}+\frac{1}{n}}\in(0,1).
 $$
\end{lemma}
 The  local  solvability and extendibility   of  the chemotaxis-growth  system \eqref{sub-log-chemo-pe} is well-established by using a suitable fixed point argument and standard parabolic regularity theory; see, for example, \cite{HW05, TW07, Win10}.
\begin{lemma}\label{local-in-time}Let $\tau\geq 0, \chi\geq 0$  and let  $\Omega\subset \mathbb{R}^n(n\geq 1)$ be a bounded domain with a smooth boundary. Suppose  that the initial data $(u_0,v_0)$ satisfies $u_0\in C(\overline{\Omega})$  and $v_0\in W^{1,r}(\Omega)$ with some $r>n$ and that $f\in W^{1,\infty}_{\mbox{loc}}(\mathbb{R})$ with $f(0)\geq 0$. Then there is a unique,  nonnegative,  classical maximal solution $(u,v)$ of the IBVP \eqref{sub-log-chemo-pe} on some maximal interval $[0, T_m)$ with $0<T_m \leq \infty$ such that
$$\ba{ll}
u\in C(\overline{\Omega}\times [0, T_m)\cap C^{2,1}(\overline{\Omega}\times (0, T_m)), \\[0.2cm]
v\in C(\overline{\Omega}\times [0, T_m))\cap C^{2,1}(\overline{\Omega}\times (0, T_m))\cap L_{\text{loc}}^\infty([0, T_m); W^{1,r}(\Omega)).
\ea $$
In particular,  if $T_m<\infty$, then
$$ \|u(\cdot, t)\|_{L^\infty(\Omega)}+\|v(\cdot, t)\|_{W^{1,r}(\Omega)}\rightarrow \infty \quad \quad \mbox{ as }  t\rightarrow T_m^{-}.
$$
\end{lemma}

\section{Blow-up prevention by sub-logistic sources}
This  section  is devoted to the  proof of Theorem \ref{main-bdd-thm}.  Unless otherwise specified, we shall assume all the conditions in Lemma \ref{local-in-time} and Theorem \ref{main-bdd-thm} are fulfilled. As usual, we begin with the $L^1$-norm of $u$, and we have the following lemma:
\begin{lemma}\label{ul1-bdd} The $L^1$-norm of $u$ and $L^2$-norm of $v$ are uniformly bounded obeying
\be\label{ul1-est-sub}
\| u(t)\|_{L^1}\leq M, \quad \forall t\in(0,T_m),
\ee
where $M$ defined by \eqref{M-def}, and, there exists $C=C(u_0,\tau v_0, |\Omega|, f)>0$ such that
\be\label{vl2-est-sub}
\| v(t)\|_{L^2}\leq C, \quad \forall t\in(0,T_m).
\ee
\end{lemma}
\begin{proof}Integrating the $u$-equation in \eqref{sub-log-chemo-pe} and using the homogeneous Neumann boundary conditions,  we obtain a Gronwall ineqaulity, for any $\eta>0$, that
\be\label{ul1-diff}
\frac{d}{dt}\int_\Omega u= \int_\Omega f(u)\leq -\eta \int_\Omega u+M_\eta|\Omega|,
\ee
which simply yields
$$
\int_\Omega u\leq \int_\Omega u_0+\frac{M_\eta}{\eta}|\Omega|.
$$
This, upon taking infimum over $\eta>0$ and recalling the definition of $M$ in \eqref{M-def}, implies the $L^1$-bound of $u$ as stated in \eqref{ul1-est-sub}. Here, due to \eqref{log-sub-con},
$$
M_\eta=\sup\Bigr\{ f(s)+\eta s:\  s>0\Bigr\}<\infty.
$$
Indeed, the definition of $\mu$ in  \eqref{log-sub-con} gives rise to
\be\label{f-asy-conl1}
\exists \hat{s}\gg 1\text{  s.t.   }  f(s)\leq -\hat{\mu}\frac{s^2}{\ln s}, \quad  \forall s\geq \hat{s},
\ee
where $\hat{\mu}=\frac{\mu}{2}$ if $0<\mu<\infty$ and $\hat{\mu}$ equals anything larger than $\chi$ if $\mu=\infty$.

This  directly entails
$$
f(s)+\eta s\leq -\hat{\mu}\frac{s^2}{\ln s}+\eta s<0, \quad \quad \forall s\geq \hat{s},
$$
which in conjunction with the fact that $f$ is bounded on any finite interval implies that $M_\eta$ is finite.

Next, since $\|u\|_{L^1}$ is bounded, the $L^1$-boundedness of $v$ follows from
$$
\tau\frac{d}{dt}\int_\Omega v+\int_\Omega v=\int_\Omega u\leq M.
$$
In the case of $\tau=0$, the $L^1$-boundedness of $u$ and the elliptic estimate applied to the $v$-equation show immediately that $\|v\|_{L^2}$ is bounded.  In the case of $\tau>0$, we (can alternatively use the Neumann heat semigroup type argument to get \eqref{vl2-est-sub}) multiply the $v$-equation by $v$ and integrate parts to deduce
\be\label{vl2-test}
\frac{\tau}{2}\frac{d}{dt}\int_\Omega   v^2+\int_\Omega   v^2+\int_\Omega |\nabla v|^2= \int_\Omega   u  v\leq \epsilon \int_\Omega   v^3+\frac{2}{3\sqrt{3\epsilon}}\int_\Omega   u^\frac{3}{2}, \quad \forall \epsilon>0,
\ee
 where we have applied the Young's inequality with epsilon:
\be\label{Young}
ab\leq \epsilon a^p+\frac{b^q}{(\epsilon p)^{\frac{q}{p}}q},  \ \ p>0, q>0, \frac{1}{p}+\frac{1}{q}=1,    \quad \forall a,b\geq 0.
\ee
By the standard  Gaglarido-Nirenberg  inequality in Lemma \ref{GN-inter} with $n=2$ and the $L^1$-bound of $v$, we conclude
$$
 \|v\|_{L^3}^3
  \leq C_{GN}(\|\nabla v\|_{L^2}^{\frac{2}{3}}\|v\|_{L^1}^\frac{1}{3}+\| v\|_{L^1})^3
\leq C\|\nabla v\|_{L^2}^2+C.
 $$
On the other hand, it is simple to see from \eqref{f-asy-conl1} that
$$
A_\epsilon:=\sup\Bigr\{ f(s)+s+\frac{2}{3\sqrt{3\epsilon}} s^\frac{3}{2}:\  s>0\Bigr\}<\infty.
$$
By taking sufficiently small $\epsilon>0$, combing these two inequalities above with \eqref{ul1-diff} and \eqref{vl2-test}, we finally derive a differential inequality as follows:
$$
\frac{d}{dt} \int_\Omega \Bigr(u+\frac{\tau}{2} v^2\Bigr)+\min\{1, \frac{2}{\tau}\}\int_\Omega \Bigr(u+\frac{\tau}{2} v^2\Bigr)\leq C_\epsilon.
$$
Solving this standard Gronwall inequality, we readily obtain the boundedness of $\|u\|_{L^1}+\|v\|_{L^2}$, and so the $L^2$-boundedness of $v$ as in \eqref{vl2-est-sub} follows.
\end{proof}
Following the common way, cf. \cite{OTYM02,  Xiangpre}, based on the $L^1$-estimate of $u$ as gained in Lemma \ref{ul1-bdd}, we proceed to show the uniform boundedness of $\|u\ln u\|_{L^1}$
\begin{lemma}\label{ulnu-bdd}  There exists $C=C(u_0,\tau v_0, |\Omega|, \chi, f)>0$ such that
\be\label{ulnu-gradvl2-est}
\left\| (u\ln u)(t))\right\|_{L^1}+\tau \chi \|\nabla v\|_{L^2}^2\leq C, \quad \quad \forall t\in(0,T_m).
\ee
\end{lemma}
\begin{proof}  We use integration by parts to compute honestly from \eqref{sub-log-chemo-pe} that
 \be \label{ulnu-diff}
 \begin{split}
\frac{d}{dt} \int_\Omega u\ln u+ 4\int_\Omega |\nabla u^\frac{1}{2} |^2&=\chi\int_\Omega \nabla u\nabla v+\int_\Omega (\ln u+1) f(u)\\
& =-\chi \int_\Omega   u\Delta  v+\int_\Omega (\ln u+1) f(u).
\end{split}
\ee
In the case of $\tau=0$, we substitute the express $\Delta v=v-u$ from the $v$-equation in \eqref{sub-log-chemo-pe} into \eqref{ulnu-diff} and utilize the nonnegativity of $u, v$ and $\chi$ to obtain
 \be \label{ulnu+gradv-diff}
  \frac{d}{dt} \int_\Omega u\ln u+ 4\int_\Omega |\nabla u^\frac{1}{2} |^2 \leq \int_\Omega   [\chi u^2+ (\ln u+1) f(u)].
\ee
In the case of $\tau>0$, we multiply the $v$-equation by $-\Delta v$ and then integrate by parts over $\Omega$ to get
\be\label{v-test}
\frac{\tau}{2}\frac{d}{dt}\int_\Omega |\nabla v|^2+\int_\Omega |\nabla v|^2+\int_\Omega |\Delta v|^2= -\int_\Omega   u\Delta  v.
\ee
We obtain through an obvious linear combination of identities \eqref{ulnu-diff} and \eqref{v-test} and a use of the Cauchy-Schawrz inequality that
\be\label{ulnu-gradv-com}
 \begin{split}
\frac{d}{dt} \int_\Omega \Bigr(u\ln u+\frac{\tau\chi}{2}|\nabla v|^2\Bigr)&+ 4\int_\Omega |\nabla u^\frac{1}{2} |^2+\chi\int_\Omega |\nabla v|^2+\chi\int_\Omega |\Delta v|^2\\
& =-2\chi \int_\Omega   u\Delta  v+\int_\Omega (\ln u+1) f(u)\\
&\leq  \chi\int_\Omega   |\Delta  v|^2+\int_\Omega \Bigr[\chi u^2+(\ln u+1) f(u)\Bigr].
\end{split}
\ee
Now, the key is to digest the last integral on its right-hand side of \eqref{ulnu+gradv-diff} and \eqref{ulnu-gradv-com}.

First, the definition of $\mu$ in  \eqref{log-sub-con}  allows us to deduce that
\be\label{mu<infty-equb3}
 \forall \epsilon \in (0, \mu),  \  \exists s_\epsilon \gg 1\text{  s.t.   }  f(s)\leq -(\mu-\epsilon)\frac{s^2}{\ln s}, \quad  \forall s\geq s_\epsilon.
\ee
Here and below, $\mu$ is understood as any finite number larger than $\chi$ in the case of $\mu=\infty$. Then it follows trivially from \eqref{mu<infty-equb3} that
$$
\chi s^2+(\ln s+1) f(s)\leq  \chi s^2-(\mu-\epsilon)\frac{(\ln s+1)}{\ln s} s^2\leq (\chi-\mu)^+s^2+\epsilon s^2, \ \ \forall s\geq s_\epsilon.
$$
Therefore, we have
\be\label{f-term-crol}
\begin{split}
&\int_\Omega \Bigr[\chi u^2+(\ln u+1) f(u)\Bigr]\\
&=\int_{\{u\leq s_\epsilon\}} \Bigr[\chi u^2+(\ln u+1) f(u)\Bigr]+\int_{\{u> s_\epsilon\}} \Bigr[\chi u^2+(\ln u+1) f(u)\Bigr]\\
&\leq \sup_{0<s<s_\epsilon}\Bigr[\chi s^2+(\ln s+1) f(s)\Bigr]|\Omega|+ (\chi-\mu)^+\int_\Omega u^2+\epsilon \int_\Omega u^2.
\end{split}
\ee
Now, we apply  the well-known two dimensional  Gaglarido-Nirenberg  inequality in Lemma \ref{GN-inter} with $n=2$ and the $L^1$-bound of $u$ in \eqref{ul1-est-sub} to estimate
\be\label{GN1-ulnu}
\begin{split}
\int_\Omega u^2=\|u^\frac{1}{2}\|_{L^4}^4&\leq C_{GN}^4 \Bigr(\|\nabla u^\frac{1}{2}\|_{L^2}^\frac{1}{2}\|u^\frac{1}{2}\|_{L^2}^\frac{1}{2}+\|u^\frac{1}{2}\|_{L^2}\Bigr)^4\\
&\leq 8C_{GN}^4\Bigr(M \|\nabla u^\frac{1}{2}\|_{L^2}^2+ M^2\Bigr),
\end{split}
\ee
where we applied the elementary inequality $(a+b)^4\leq 2^3(a^4+b^4)$ for all $ a, b\geq 0$.

Next, noticing that
\be\label{ulnu-by-u2}
u\ln u\leq \epsilon u^2+L_\epsilon, \quad L_\epsilon=\sup_{s>0}\{s\ln s-\epsilon s^2:\  s>0\}<\infty,
\ee
we conclude from \eqref{f-term-crol}, \eqref{GN1-ulnu} and \eqref{ulnu-by-u2} that
\be\label{f-term-crol-fin}
 \int_\Omega \Bigr[\chi u^2+(\ln u+1) f(u)+u\ln u\Bigr]\leq 8MC_{GN}^4\Bigr[(\chi-\mu)^++2\epsilon\Bigr]\int_\Omega |\nabla u^\frac{1}{2}|^2+N_\epsilon,
\ee
where
$$
N_\epsilon=\sup_{0<s<s_\epsilon}\Bigr[\chi s^2+(\ln s+1) f(s)\Bigr]|\Omega|+L_\epsilon|\Omega|+8M^2C_{GN}^4\Bigr[(\chi-\mu)^++2\epsilon\Bigr]<\infty.
$$
Now, thanks to \eqref{sub-log-general}, we can fix, for instance,  $\epsilon$ according to
\be\label{epsilon-0}
\begin{split}
\epsilon=\epsilon_0&= \frac{1}{2}\min\left\{\mu,  \frac{1}{2MC_{GN}^4}-(\chi-\mu)^+\right\}>0 \\
&\Longleftrightarrow 8MC_{GN}^4\Bigr[(\chi-\mu)^++2\epsilon_0\Bigr]\leq 4,
\end{split}
\ee
and then, in the case of $\tau=0$,  combing \eqref{f-term-crol-fin} with \eqref{ulnu+gradv-diff},  we readily  infer a differential inequality as follows:
$$
  \frac{d}{dt} \int_\Omega u\ln u + \int_\Omega u\ln u \leq N_{\epsilon_0};
$$
this easily entails that
$$
\int_\Omega   u\ln u \leq \Bigr(\int_\Omega   u_0\ln u_0\Bigr) +N_{\epsilon_0}.
$$
Hence,   the fact $-s\ln s\leq e^{-1}$ for all $s>0$ further entails
\be\label{ulnu-sub-bdd1}
\int_\Omega|u\ln u| =\int_\Omega u\ln u-2\int_{\{u\leq 1\}} u\ln u\leq \int_\Omega u_0\ln u_0 +(N_{\epsilon_0}+2e^{-1}|\Omega|).
\ee
In the case of $\tau>0$,  collecting \eqref{epsilon-0}, \eqref{f-term-crol-fin} and \eqref{ulnu-gradv-com}, we deduce another  key differential inequality:
$$
\frac{d}{dt} \int_\Omega \Bigr(u\ln u+\frac{\tau\chi}{2}|\nabla v|^2\Bigr)+\min\{1, \frac{2}{\tau}\}\int_\Omega \Bigr(u\ln u+\frac{\tau\chi}{2}|\nabla v|^2\Bigr)\leq N_{\epsilon_0},
$$
which yields simply
$$
\int_\Omega \Bigr(u\ln u+\frac{\tau\chi}{2}|\nabla v|^2\Bigr)\leq \int_\Omega \Bigr(u_0\ln u_0+\frac{\tau\chi}{2}|\nabla v_0|^2\Bigr)+\frac{N_{\epsilon_0}}{\min\{1,\frac{2}{\tau}\}}.
$$
This together with \eqref{ulnu-sub-bdd1} gives our desired estimate in \eqref{ulnu-gradvl2-est}.
\end{proof}
For the standard logistic source  $f(s)=as-bs^2$, based on the $L^1$-boundedness of $u\ln u$,  there are two common methods, cf \cite{OTYM02, TW12, Xiangpre},  to obtain $L^2$-boundedness $u$. To make our argument self-contained and for completeness, we here sketch these two methods for sub-logistic sources satisfying \eqref{log-sub-con}.

\begin{lemma} \label{pe-case} For $\tau=0$, there exists $C=C(u_0, |\Omega|, \chi,  f)>0$, and, for any $q\in (0,\infty)$, there exists $C_q=C(q, u_0, |\Omega|,\chi, f)>0$ such that
 \be\label{ul2-pe-case}
 \|u(\cdot, t)\|_{L^2}\leq C, \quad \quad \|\nabla v(\cdot, t)\|_{L^q}\leq C_q, \ \ \forall t\in(0, T_m).
 \ee
\end{lemma}
\begin{proof}
By testing the $u$-equation by $u$ and using the $v$-equation with $\tau=0$ in the chemotaxis-growth model \eqref{sub-log-chemo-pe}, upon integration by parts, we find that
 \be \label{up-diff}
 \frac{d}{dt} \int_\Omega u^2+ 2\int_\Omega |\nabla u |^2\leq \chi\int_\Omega u^3+2\int_\Omega u f(u).
\ee
Thanks to the uniform  boundedness of $\|u\ln u\|_{L^1}$ by \eqref{ulnu-gradvl2-est} of Lemma \ref{ulnu-bdd} and the extended version of Gaglarido-Nirenberg inequality involving logarithmic functions  from  \cite[Lemma A.5]{TW14-JDE}, we can easily deduce  that
\be\label{ul3-bdd by}
\int_\Omega  u^3\leq \eta \int_\Omega |\nabla u |^2+C_\eta, \quad \quad \forall \eta>0.
\ee
On the other hand, it follows from \eqref{f-asy-conl1} that $\sup\{ 2sf(s)+s^2:\  s>0\}<\infty$.
Hence, by taking $\eta\in (0, \frac{2}{\chi})$ in \eqref{ul3-bdd by}, we deduce from \eqref{up-diff} an ODI for $\int_\Omega u^2$:
$$
\frac{d}{dt} \int_\Omega u^2+  \int_\Omega u^2\leq C(u_0, |\Omega|, \chi,  f),
$$
which quickly gives rise to the $L^2$-boundedness of $u$ as in \eqref{ul2-pe-case}.

 Then the $W^{2,p}$-elliptic estimate applied to $-\Delta v+v=u$ implies the boundedness of $\|v\|_{W^{2,2}}$, and then  the Sobolev embedding gives  the boundedness of $v$ as described  in the second part of \eqref{ul2-pe-case}.
\end{proof}
\begin{lemma} \label{pp-case} For  $\tau>0$, there exists $C=C(u_0, v_0, |\Omega|, \chi, \tau, f)>0$ such that
 \be\label{ul2-pp-case}
 \|u(\cdot, t)\|_{L^2}+\|\nabla v(\cdot,t)\|_{L^4}+\|\Delta  v(\cdot,t)\|_{L^2}\leq C, \ \ \forall t\in(0, T_m).
 \ee
\end{lemma}
\begin{proof} We multiply  the  $u$- equation in \eqref{sub-log-chemo-pe} by $u$  and integrate  by parts to infer  from H\"{o}lder's  inequality  that
\be \label{u-lp-int}
\begin{split}
\frac{1}{2}\frac{d}{dt} \int_\Omega u^2+ \int_\Omega |\nabla u |^2& =\chi  \int_\Omega u\nabla u \nabla v+\int_\Omega uf(u)\\
&\quad = - \frac{\chi}{2} \int_\Omega  u^2 \Delta  v +\int_\Omega uf(u).   \end{split}
\ee
Now, we have two choices to obtain the $L^2$-estimate of $u$ by coupling \eqref{u-lp-int} with two energy identities associated with the $v$-equation. The first choice is to multiply the $v$-equation by $\Delta^2v=\Delta(\Delta v)$ and then integrate over $\Omega$ by parts to obtain that
\be\label{delav-diff}
\frac{1}{2}\frac{d}{dt} \int_\Omega |\Delta v|^2+ \int_\Omega |\Delta v |^2+\int_\Omega |\nabla \Delta v |^2=-\int_\Omega \nabla u \nabla \Delta v.
\ee
The second choice is to take gradient of the $v$-equation  and then multiply it by $\nabla v|\nabla v|^2$ and, finally integrate by parts to see that
\be\label{gradvl4}\begin{split}
&\frac{1}{2}\frac{d}{dt} \int_\Omega |\nabla v|^4+\int_\Omega |\nabla |\nabla v|^2|^2+2\int_\Omega  |\nabla v|^{2}|D^2v|^2+2\int_\Omega |\nabla v|^{4}\\
&=\int_{\partial\Omega}  |\nabla v|^{2}\frac{\partial}{\partial \nu} |\nabla v|^2-\int_\Omega  u\Delta  v|\nabla v|^{2}-\int_\Omega  u\nabla  v\nabla |\nabla v|^{2}.  \end{split}
\ee
{\bf Method I:}  We combine \eqref{u-lp-int} with \eqref{delav-diff} to get, for any $\epsilon>0$, that
\be \label{u-lp-deltav-com1}
\begin{split}
\frac{1}{2}\frac{d}{dt} \int_\Omega \Bigr(u^2+ |\Delta v|^2\Bigr)&+\int_\Omega |\nabla u |^2+ \int_\Omega |\Delta v |^2+\int_\Omega |\nabla \Delta v|^2\\
& = \frac{\chi}{2} \int_\Omega  u^2 \Delta  v -\int_\Omega \nabla u \nabla \Delta v+\int_\Omega uf(u)\\
&\leq \epsilon \int_\Omega  |\Delta  v|^3+\frac{\chi^\frac{3}{2}}{3\sqrt{6\epsilon}}\int_\Omega u^3 +\frac{1}{2}\int_\Omega |\nabla u|^2\\
&\ \ \ +\frac{1}{2}\int_\Omega |\nabla \Delta v|^2++\int_\Omega uf(u),  \end{split}
\ee
where we have applied the Young's inequality with epsilon \eqref{Young}.

 Applying the Gaglarido-Nirenberg interpolation inequality  with $n=2$, Sobolev interpolation inequality and  the boundedness of $\| v\|_{H^1}$ implied by \eqref{vl2-est-sub} and  \eqref{ulnu-gradvl2-est}, we derive (see details, for instance, in  \cite{OTYM02, Xiangpre}) that
 \be\label{delta vl3-bdd by}
\int_\Omega  |\Delta v|^3\leq C\int_\Omega |\nabla\Delta v|^2+C.
\ee
Now, since $sup\{sf(s):  s>0\}<\infty$ clearly implied by \eqref{log-sub-con}, based on \eqref{delta vl3-bdd by}, \eqref{ul3-bdd by} and \eqref{u-lp-deltav-com1}, one can easily deduce a Gronwall inequality of the form:
$$
\frac{d}{dt} \int_\Omega \Bigr(u^2+ |\Delta v|^2\Bigr)+ \int_\Omega \Bigr(u^2+ |\Delta v|^2\Bigr)\leq C(u_0, v_0, |\Omega|, \chi, \tau, f),
$$
yielding directly the uniform boundedness of $\|u\|_{L^2}+\|\Delta v\|_{L^2}$ as stated in \eqref{ul2-pp-case}.

{\bf Method II:}  We combine \eqref{u-lp-int} with \eqref{gradvl4} to derive that
\be \label{u-lp-deltav-com}
\begin{split}
\frac{1}{2}\frac{d}{dt} \int_\Omega \Bigr(u^2+ |\nabla v|^4\Bigr)&+\int_\Omega |\nabla u |^2+\int_\Omega |\nabla |\nabla v |^2|^2\\
&\ \ + 2\int_\Omega |\nabla v |^2|D^2v|^2+2\int_\Omega |\nabla v|^4\\
& = \chi \int_\Omega u\nabla  u \nabla   v -\int_\Omega  u\Delta  v|\nabla v|^{2} -\int_\Omega  u\nabla  v\nabla |\nabla v|^{2}\\
&\quad \quad +\int_{\partial\Omega}  |\nabla v|^{2}\frac{\partial}{\partial \nu} |\nabla v|^2+\int_\Omega uf(u).  \end{split}
\ee
Next, we shall estimate the integrals on the right in terms of the dissipation terms on its left using the very common ideas:
\be\label{com-est-1}
\begin{split}
&\chi \int_\Omega u\nabla  u \nabla   v -\int_\Omega  u\Delta  v|\nabla v|^{2}-\int_\Omega  u\nabla  v\nabla |\nabla v|^{2}\\
 &\leq \frac{1}{2} \int_\Omega |\nabla  u|^2+\frac{\chi^2}{2}\int_\Omega u^2| \nabla   v|^2+\int_\Omega   |\Delta v|^2|\nabla v|^2+\frac{1}{4}\int_\Omega u^2| \nabla   v|^2\\
 &\quad + \frac{1}{2}\int_\Omega |\nabla |\nabla v|^2|^2+\frac{1}{2}\int_\Omega u^2| \nabla   v|^2\\
 &\leq \frac{1}{2} \int_\Omega |\nabla  u|^2+\frac{1}{2}\int_\Omega |\nabla |\nabla v|^2|^2+2\int_\Omega   |D^2 v|^2|\nabla v|^2+(1+\chi^2)\int_\Omega u^2| \nabla   v|^2
 \end{split}
 \ee
 and the Young's inequality with epsilon \eqref{Young}  shows
\be\label{com-est-2}
\int_\Omega u^2| \nabla   v|^2\leq \epsilon \int_\Omega  |\nabla   v|^6+\frac{2}{3\sqrt{3\epsilon}}\int_\Omega u^3,\quad \quad \forall \epsilon>0.
\ee
In view of the  boundedness of $\|\nabla v\|_{L^2}$ by \eqref{ulnu-gradvl2-est}, the 2-D GN inequality entails
\be\label{com-est-3}
\begin{split}
 \int_\Omega  |\nabla   v|^6=\||\nabla v|^2\|_{L^3}^3&\leq\Bigr( C_{GN}(|\nabla |\nabla v|^2\|_{L^2}^\frac{2}{3}\||\nabla v|^2\|_{L^1}^\frac{1}{3}+\||\nabla v|^2\|_{L^1})\Bigr)^3\\
 &\leq C|\nabla |\nabla v|^2\|_{L^2}^2+C.
 \end{split}
\ee
As for the boundary integral in \eqref{u-lp-deltav-com}, there are a couple of known ways to bound it in terms of  the boundedness of $\|\nabla v\|_{L^2}$, cf.  \cite{ISY14,TWZAMP, Xiangpre2};  the final outcome is
\be\label{com-est-4}
\begin{split}
\int_{\partial\Omega}  |\nabla v|^{2}\frac{\partial}{\partial \nu} |\nabla v|^2&\leq \epsilon \int_\Omega |\nabla |\nabla v|^2|^2+C_\epsilon \Bigr(\int_{\Omega} |\nabla v|^2\Bigr)^2\\
&\leq \epsilon \int_\Omega |\nabla |\nabla v|^2|^2+C_\epsilon,\quad \quad  \forall \epsilon>0.
\end{split}
\ee
Inserting the estimates \eqref{com-est-1}, \eqref{com-est-2}, \eqref{com-est-3}, \eqref{com-est-4} and \eqref{ul3-bdd by} into \eqref{u-lp-deltav-com} and then choosing sufficiently small $\epsilon>0$, as before, we obtain an ODI as follow:
$$
\frac{d}{dt} \int_\Omega \Bigr(u^2+ |\nabla v|^4\Bigr)+ \int_\Omega \Bigr(u^2+ |\nabla v|^4\Bigr)\leq C(u_0, v_0, |\Omega|, \chi, \tau, f),
$$
which directly establishes the uniform boundedness of $\|u\|_{L^2}+\|\nabla v\|_{L^4}$.
\end{proof}
\begin{proof}[Proof of Theorem \ref{main-bdd-thm}] In light of  the uniform $L^2$-boundedness of $u$ provided by Lemmas \ref{pe-case} and \ref{pp-case}, the   quite known $L^{\frac{n}{2}+}$-boundedness criterion with $n=2$ in \cite{BBTW15, Xiangpre} obtained via Moser type iteration technique shows  $T_m=\infty$  and the uniform  boundedness as stated in \eqref{bdd-linfty}. Notice that $(u(\cdot,\sigma), v(\cdot,\sigma)\in (C^2(\overline{\Omega}))^2$ for any $\sigma>0$. Whence,   performing a small time shift and treat $t=\sigma$  as the new "initial time" and replacing $(u_0,v_0)$ by $(u(\cdot,\sigma), v(\cdot,\sigma))$, applying the same $L^{\frac{n}{2}+}$-boundedness criterion with $n=2$, we conclude the uniform boundedness as stated in  \eqref{bdd-linfty1}.
\end{proof}


\begin{thebibliography}{99}
 \footnotesize


\bibitem{BBTW15} N. Bellomo, A. Bellouquid, Y. Tao, M. Winkler Toward a mathematical theory of Keller-Segel models of pattern formation in biological tissues, Math. Models Methods Appl. Sci. 25 (2015), 1663--1763.



 \bibitem{Cao15} X. Cao, Global bounded solutions of the higher-dimensional Keller-Segel system under smallness conditions in optimal spaces, Discrete Contin. Dyn. Syst.  35  (2015),   1891--1904.




 \bibitem{FLP07} E. Feireisl, P. Laurencot and H. Petzeltova, On convergence to equilibria for the Keller-Segel chemotaxis model, J. Differential Equations, 236 (2007), 551--569.


 \bibitem{Fried} A. Friedman,  Partial differential equations. Holt, Rinehart and Winston, New York-Montreal, Que.-London, 1969.


 \bibitem{FW14} K. Fujie, M.  Winkler and T.  Yokota,  Blow-up prevention by logistic sources in a parabolic-elliptic Keller-Segel system with singular sensitivity,  Nonlinear Anal.  109  (2014), 56--71.


\bibitem{HZ16} X. He and S. Zheng,  Convergence rate estimates of solutions in a higher dimensional chemotaxis system with logistic source,  J. Math. Anal. Appl.  436  (2016),    970--982.

  \bibitem{HP04}   T.  Hillen and A. Potapov, The one-dimensional chemotaxis model: global existence and asymptotic profile,  Math. Methods Appl. Sci. 27 (2004),  1783--1801.


  \bibitem{Hi} T. Hillen and K. Painter, A user's  guide to PDE models for chemotaxis,  J. Math. Biol., 58 (2009), 183--217.

  \bibitem{HP11} T. Hillen and  K. Painter, Spatio-temporal chaos in a chemotaxis model, Phys. D 240 (2011),  363--375.
  \bibitem{HW01} D. Horstmann and G.  Wang, Blow-up in a chemotaxis model without symmetry assumptions, European J. Appl. Math. 12 (2001), 159--177.

  \bibitem{Ho1} D. Horstmann, From 1970 until now: the Keller-Segal model in chaemotaxis and its consequence I,  Jahresber   DMV, 105 (2003), 103--165.

   \bibitem{HW05} D. Horstmann and M.  Winkler,  Boundedness vs. blow-up in a chemotaxis system, J. Differential Equations 215 (2005),  52--107.



  \bibitem{HT16} B. Hu and Y. Tao,  Boundedness in a parabolic-elliptic chemotaxis-growth system under a critical parameter condition, Appl. Math. Lett. 64 (2017), 1--7.

\bibitem{ISY14} S. Ishida, K. Seki and T.  Yokota, Boundedness in quasilinear Keller-Segel systems of parabolic-parabolic type on non-convex bounded domains,  J. Differential Equations  256  (2014),  no. 8, 2993--3010.

\bibitem{Ke} E. Keller and L. Segel, Initiation of slime mold aggregation viewed as an instability, J. Theoret Biol.,  26  (1970), 399--415.

\bibitem{Ke1} E. Keller and L. Segel,  Model for chemotaxis,  J. Theor. Biol.,  30 (1971), 225--234.



\bibitem{KS16} K. Kang and A.  Stevens,  Blowup and global solutions in a chemotaxis-growth system,
Nonlinear Anal. 135 (2016), 57--72.


\bibitem{La15-JDE}J. Lankeit, Eventual smoothness and asymptotics in a three-dimensional chemotaxis system with logistic source,  J. Differential Equations  258  (2015), 1158--1191.

\bibitem{La15-DCDS} J. Lankeit, Chemotaxis can prevent thresholds on population density,  Discrete Contin. Dyn. Syst. Ser. B  20  (2015),  1499--1527.



  \bibitem{Na01} T. Nagai, Blowup of nonradial solutions to parabolic-elliptic systems modeling chemotaxis in two-dimensional domains, J. Inequal. Appl. 6 (2001), 37--55.

 \bibitem{NSY97} T. Nagai,  T. Senba and K. Yoshida,
Application of the Trudinger-Moser inequality to a parabolic system of chemotaxis,
Funkcial. Ekvac.  40  (1997),  411--433.





\bibitem{Nirenberg66} L. Nirenberg,   An extended interpolation inequality,
Ann. Scuola Norm. Sup. Pisa. (3) 20 1966,  733-737.


\bibitem{OTYM02} K. Osaki,T.  Tsujikawa, A. Yagi, M.  Mimura, Exponential attractor for a chemotaxis-growth system of equations, Nonlinear Anal. 51, 119-144 (2002).

 \bibitem{OY01}  K.  Osaki and A. Yagi, Finite dimensional attractor for one-dimensional Keller-Segel equations, Funkcial. Ekvac. 44 (2001), 441--469.


 \bibitem{SS01} T.  Senba and T. Suzuki, Parabolic system of chemotaxis: blowup in a finite and the infinite time, Methods Appl. Anal. 8 (2001), 349--367.


 \bibitem{TW12} Y. Tao and M. Winkler, Boundedness in a quasilinear parabolic-parabolic Keller-Segel system with subcritical sensitivity,  J. Differential Equations 252  (2012),  692--715.

 \bibitem{TW14-JDE} Y.  Tao and M. Winkler, Energy-type estimates and global solvability in a two-dimensional chemotaxis-haptotaxis model with remodeling of non-diffusible attractant, J. Differential Equations 257 (2014), 784--815.

 \bibitem{TWZAMP} Y.  Tao and M. Winkler,  Boundedness and decay enforced by quadratic degradation in a three-dimensioanl chemotaxis-fluid system, Z. Angew. Math. Phys.  66  (2015),  2555--2573.

 \bibitem{TW15-JDE}Y.  Tao and M. Winkler, Persistence of mass in a chemotaxis system with logistic source,   J. Differential Equations  259  (2015),  6142--6161.


 \bibitem{TW07} J. Tello and M.  Winkler,  A chemotaxis system with logistic source, Comm. Partial Differential Equations 32 (2007),  849--877.


 \bibitem{WXpre} Z. Wang and T. Xiang,  A  class of chemotaxis systems with   growth  source  and nonlinear secretion, arXiv:1510.07204, 2015.

\bibitem{Win10-JDE} M. Winkler,  Aggregation vs. global diffusive behavior in the higher-dimensional Keller-Segel model, J. Differential Equations, 248 (2010), 2889--2905.


\bibitem{Win10}M. Winkler,  Boundedness in the higher-dimensional parabolic-parabolic chemotaxis system with logistic source,  Comm. Partial Differential Equations 35  (2010), 1516--1537.


\bibitem{Win11} M. Winkler, Blow-up in a higher-dimensional chemotaxis system despite logistic growth restriction, J. Math. Anal. Appl. 384  (2011), 261--272.

\bibitem{Win13} M. Winkler, Finite-time blow-up in the higher-dimensional parabolic-parabolic Keller-Segel system, J. Math. Pures Appl. 100  (2013), 748--767.

\bibitem{Win14} M. Winkler,  Global asymptotic stability of constant equilibria in a fully parabolic chemotaxis system with strong logistic dampening, J. Differential Equations  257  (2014), 1056--1077.

\bibitem{Win14-JNS} M. Winkler,  How far can chemotactic cross-diffusion enforce exceeding carrying capacities? J. Nonlinear Sci.  24  (2014),  809--855.



\bibitem{Win17-DCDSB}M. Winkler, Emergence of large population densities despite logistic growth restrictions in fully parabolic chemotaxis systems, Discrete Contin. Dyn. Syst. Ser. B 22 (2017), 2777--279.


\bibitem{Xiang14}   T. Xiang, On effects of sampling radius for the nonlocal Patlak-Keller-Segel chemotaxis model,  Discrete Contin. Dyn. Syst. 34 (2014), 4911-4946.



\bibitem{Xiangpre}T. Xiang, Boundedness and global existence in the higher-dimensional parabolic–parabolic chemotaxis system with/without growth source,
J. Differential Equations 258 (2015),  4275--4323.


\bibitem{Xiangpre2} T.Xiang, How strong  a logistic damping can  prevent blow-up for the minimal  Keller-Segel chemotaxis system?  J. Math. Anal. Appl. 459 (2018),  1172--1200.

  \bibitem{YCJZ15}C. Yang, X.  Cao, Z. Jiang and S. Zheng,  Boundedness in a quasilinear fully parabolic Keller-Segel system of higher dimension with logistic source, J. Math. Anal. Appl.  430  (2015),  585--591.


 \bibitem{ZZ17-ZAMP} X. Zhao and S.  Zheng,   Global boundedness to a chemotaxis system with singular sensitivity and logistic source,  Z. Angew. Math. Phys. 68 (2017), no. 1, Art. 2, 13 pp. .

\end{thebibliography}
\end{document}